\documentclass[10pt,a4wide]{article}

\usepackage{amsmath} 
\usepackage{amssymb}  
\usepackage{amsthm}
\usepackage{caption}
\usepackage{subcaption}
\usepackage{dsfont}
\usepackage{graphicx}
\usepackage{psfrag}
\usepackage{hyperref}
\usepackage{color}
\usepackage{mathdots}
\usepackage{pst-all}
\usepackage{pstricks}
\usepackage{pstricks-add}
\usepackage{pst-plot}
\usepackage{pst-infixplot}
\usepackage{pstricks-add}
\usepackage{wrapfig}
\usepackage{tikz}
\usepackage{capt-of}
\usepackage{float}

\usepackage{enumitem}

\usepackage{mathtools}
\usepackage{xspace}

\usepackage[hyphenbreaks]{breakurl}
\usepackage{tikz}
\usetikzlibrary{shapes,arrows}
\usetikzlibrary{arrows,calc}

\tikzset{
    block/.style = {draw, rectangle, 
        minimum height=1cm, 
        minimum width=2cm},
    input/.style = {coordinate,node distance=1cm},
    output/.style = {coordinate,node distance=4cm},
    arrow/.style={draw, -latex,node distance=2cm},
    pinstyle/.style = {pin edge={latex-, black,node distance=2cm}},
    sum/.style = {draw, circle, node distance=1cm}
}

\makeatletter
\newcommand{\specificthanks}[1]{\@fnsymbol{#1}}
\makeatother

\newtheorem{theorem}{Theorem}[section]
\newtheorem{lemma}[theorem]{Lemma}
\newtheorem{proposition}[theorem]{Proposition}

\theoremstyle{definition}
\newtheorem{definition}[theorem]{Definition}
\newtheorem{example}[theorem]{Example}
\theoremstyle{remark}
\newtheorem{remark}[theorem]{Remark}
\newtheorem{assumption}[theorem]{Assumption}

\renewcommand{\Re}{\mathrm{Re\,}}

\renewcommand{\Im}{\mathrm{Im\,}}

\renewcommand{\>}{\rangle}
\renewcommand{\H}{\mathcal{H}}
\newcommand{\BL}{\mathcal{L}}
\newcommand{\R}{\mathbb{R}}		
\newcommand{\N}{\mathbb{N}}	
\newcommand{\C}{\mathbb{C}}

\newcommand{\mS}{\mathcal{S}}
\newcommand{\mA}{\mathcal{A}}
\newcommand{\T}{\mathcal{T}}	

\newcommand{\D}{\mathcal{D}}
\newcommand{\norm}[1]{\left\| #1 \right\|}	
\newcommand{\norms}[1]{\Bigl\|#1\Bigr\|} 
\newcommand{\normk}[1]{\bigl\|#1\bigr\|} 

\newcommand{\abs}[1]{\left| #1 \right|}	
\renewcommand{\phi}{\varphi}

\newcommand{\riesz}{discrete Riesz spectral operator\xspace}

\newcommand{\rieszp}{\riesz}

	\newcommand{\ran}{\mathop{\mathrm{ran}}}
	\DeclareMathOperator{\Span}{span}
\DeclareMathOperator{\diag}{diag}
\makeatletter
\renewcommand*{\@fnsymbol}[1]{\ifcase#1\or*\else\@arabic{\numexpr#1-1\relax}\fi}
\makeatother

\title
{Riesz bases of port-Hamiltonian systems\thanks {
Support by Deutsche Forschungsgemeinschaft (Grant JA 735/13-1) 
is gratefully acknowledged.}}

\author{Birgit Jacob\thanks{University of Wuppertal, School of Mathematics and Natural Sciences, IMACM
Gau\ss stra\ss e 20,
D-42119 Wuppertal, Germany, $\{$bjacob,julia.kaiser$\}$@uni-wuppertal.de}
\and Julia T.~Kaiser\footnotemark[2] \and Hans Zwart\thanks{University of Twente, Department of Applied Mathematics,
7500 AE Enschede, The Netherlands, h.j.zwart@utwente.nl}\textsuperscript{\phantom{2},\specificthanks{0}}\thanks{Department of Mechanical Engineering, Technische Universiteit Eindhoven, 5600 MB Eindhoven, The Netherlands}}
\date{}

\begin{document}
\maketitle
\thispagestyle{empty}
\begin{abstract}
 The  location of the spectrum and the Riesz basis property of  well-posed homogeneous infinite-dimensional linear port-Hamiltonian systems on a 1D spatial domain are studied. It is shown that  the Riesz basis property is equivalent to the fact that system operator generates a strongly continuous group. Moreover, in this situation the spectrum consists of eigenvalues only, located in a strip parallel to the imaginary axis  and  they can decomposed into finitely many sets having each a uniform gap. 
\end{abstract}


\section{Introduction}\label{intro}

It is well-known that the eigenvectors of a compact self-adjoint operator form an orthonormal basis of the underlying Hilbert space. 
In the 1960s Dunford and Schwarz \cite{DunSch71} introduced the more general notion of spectral 
operators. Further, Curtain \cite{Cur84} analysed discrete spectral operators, i.e., spectral operators with  compact resolvent, and the class of Riesz spectral operators was formulated in \cite{CurZwa95} and extended in \cite{GuoZwa01} 
to characterize also operators with multiple eigenvalues. For Riesz spectral operator its eigenvectors still form a basis, but this basis is assumed to be Riesz basis. Since a Riesz basis is isomorphic to an orthonormal basis, many of the nice properties of compact self-adjoint operators cary over to Riesz spectral operators. For instance, solutions of the abstract differential equation $\dot{x}(t) = A x(t) + B u(t)$, with $A$ a Riesz spectral operator, can still described by an eigenfunction expansion of non-harmonic Fourier series. This enables that many properties of these infinite-dimensional systems such as stability, stabilizability and controllability can be characterized in an elegant manner, see e.g.\ \cite{CurZwa95,CuZw20}.


In this article, we investigate the Riesz basis property of a special class of infinite-dimensional systems, namely port-Hamiltonian systems on a 1D spatial domain. Here by the Riesz basis property we mean that the associated system operator is a \riesz, see Definition \ref{defDRSO}. 
Port-based network modelling of complex physical systems leads to port-Hamiltonian systems. For 
finite-dimensional systems there is by now a well-established theory
\cite{vanDerSchaft06,EbMS07,DuinMacc09}. 
The port-Hamiltonian approach has been extended to the  
infinite-dimensional situation by a functional analytic approach in \cite{Vil07,ZwaGorMasVil10,JacZwa12,JacMorZwa15,Aug16,JacZwa19}. This approach has been successfully used to derive simple verifiable conditions for 
well-posedness \cite{GorZwaMas05,Vil07,ZwaGorMasVil10,JacZwa12,JacMorZwa15,JacKai18}, stability 
\cite{JacZwa12,AugJac14} and stabilization \cite{RaZwGo17,RaGoMaZw14,AugJac14,SchZwa18} and robust regulation \cite{HuPa}. 
The study of the Riesz basis property for infinite-dimensional port-Hamiltonian systems has started with the thesis by Villegas \cite[Chapter 4]{Vil07}.
Using results on first order eigenvalue problems by Tretter \cite{Tret00a,Tret00b}, he obtained a sufficient condition. However, it is not easy to see when this technical  sufficient condition is satisfied.

Many systems have a Riesz basis of eigenfunctions, see e.g.\ \cite{GuoXu04,XuGuo03}. In the monograph \cite[Section 4.3]{GuoWan19} B.-Z.\ Guo and J.-M.\ Wang study the Riesz basis property for a closely related class of systems, that is, hyperbolic systems of the form $\frac{\partial x}{\partial t} = K(\zeta) \frac{\partial x}{\partial \zeta} + C(\zeta) x$ with $K$ and $C$ diagonal. Note, that (almost) every port-Hamiltonian system on a one dimensional spatial domain can be transformed into a hyperbolic system of this form.  However, in general not with a diagonal $C$. Furthermore, the boundary conditions will be more general, see \cite{ZwaGorMasVil10,JacZwa12} or the proof of Lemma \ref{lsgode}. Our main result thus generalizes their theorem \cite[Theorem 4.11]{GuoWan19}. We remark, that in this situation the notions of Riesz basis of subspaces and Riesz basis with parentheses are equivalent. Moreover, in \cite{XuWei11} the Riesz basis property is investigated for operators perturbed by output feedback.  

Our main result shows that a linear infinite-dimensional port-Hamiltonian system on a 1D spatial domain has the Riesz basis property if and only if the system operator generates a strongly continuous group. Important to note here is that we don't need constant coefficients, nor extra assumption. Since the group property is equivalent to a simple matrix condition, our results enable us to very quickly check whether the Riesz basis property holds, see also our example section.  
 Our proof combines methods from complex analysis, differential equations and mathematical systems theory. In particular, we use the fact that every well-posed port-Hamiltonian control system \eqref{eqn:pde} is exactly controllable in finite time. We refer the reader to Section \ref{sec:PHS} for the definition of exact controllability.


\noindent
{\bf Notation:}  $X$ and $Y$ are complex and separable Hilbert space. We denote the space of all bounded linear operators from $X$ to $Y$ by $\BL(X,Y)$. To shorten notation we write $\BL(X):=\BL(X,X).$ 
Throughout this article we assume that
$A:{\cal D}(A)\subset X\rightarrow X$ is a closed, densely defined,
linear operator.
The set $\sigma(A)$ refers to the spectrum of $A$  and $\sigma_p(A)$ to its point spectrum. 
We denote by $\rho(A):=\mathbb C\backslash \sigma(A)$ the resolvent set of  $A$, and
for each $s\in\rho(A)$ we denote the resolvent of $A$ by $(s-A)^{-1}:=(sI-A)^{-1}$.
 The growth bound of a $C_0$-semigroup $(T(t))_{t\ge 0}$ with generator $A$ is denoted by $\omega_0(A)$ and $s(A):=\sup_{s \in \sigma(A)}\Re s $ is the  \emph{spectral bound}  of $A$. 
$H^1((0,1);\mathbb C^d)$ denotes the first order Sobolev space on the interval $(0,1)$.

\section{Main Result}\label{sec2}

In this section we formulate our main result. We start with  the definition of discrete Riesz spectral operators.
\begin{definition}\label{sp}
For an operator $A$ on $X$ we call  $\gamma \subset \sigma(A)$ a \emph{compact spectral set} if $\gamma$ is a compact subset of $\C$ which is open and closed in $\sigma(A)$.
The \emph{spectral projection} on the spectral subset $\gamma$ is defined as
\begin{equation*}
{E}(\gamma)=\frac{1}{2\pi i}\int_{\Gamma}(s-A)^{-1} ds,
\end{equation*}
where $\Gamma$ is a closed Jordan curve containing every point of $\gamma$ and no point of $\sigma(A)\setminus \gamma$.
\end{definition}

In this article operators with compact resolvent are of  particular interest.
The spectrum of these operators is a denumerable set of points with no finite accumulation point, c.f.~\cite[Lemma XIX.2]{DunSch71}. Furthermore, every point in the spectrum is an eigenvalue which has finite algebraic as well as  finite geometric multiplicity, c.f. \cite[Theorem XV.2.3]{GohGolKaa90}. If $(s_n)_{n\in\mathbb N}$ is the spectrum of an operator with compact resolvent we write $E_n:={E}((s_n)), n\in \N$.
 %
%
\begin{definition}\label{defDRSO}
An operator $A$ with compact resolvent and $\sigma(A)=(s_n)_{n\in \N}$ is a \emph{discrete Riesz spectral operator},
if
\begin{enumerate}
	\item for every $n\in \N$ there exists $N_n\in \BL(X)$  
	such that $AE_n = (s_n  + N_n)E_n$, 
	\item the sequence of closed subspaces $(E_n (X))_{n \in \N}$  is a \emph{Riesz basis of subspaces of 
$X$}, that is,  
$\Span (E_n (X))_{n \in \N}$ is dense and
there 
exists an isomorphism $T\in \BL(X)$, such that $(T E_n (X))_{n\in\N}$ is system of pairwise orthogonal subspaces of $X$.
	\item $N:= \sum_{n\in \N} N_n$ is bounded and nilpotent.
\end{enumerate}
\end{definition}

\begin{remark}
Discrete Riesz spectral operator are spectral operators in the sense of Dunford and Schwartz \cite{DunSch71}.  However, we additionally assume that the operator has a compact resolvent, which are discrete operators in the sense of Dunford and Schwartz \cite[Definition XIX.1]{DunSch71}.
Every discrete Riesz spectral operator is a Riesz spectral operator in the sense of Guo and Zwart \cite{GuoZwa01}  and these are again spectral operator in the sense of Dunford and Schwartz. In Curtain and Zwart \cite{CurZwa95} a slightly stronger notion is considered, where all eigenvalues have to be simple. However, they do not require that the operator has a compact resolvent. 
\end{remark}

\begin{remark}
If a sequence of vectors $(x_n)_{n \in \N}$ in  $X$ is a \emph{Riesz basis of $X$}, that is,  there 
is an isomorphism $T\in \BL(X)$, such that $(Tx_n)_{n\in N}$ is an 
orthonormal 
basis of 
$X$, then clearly $\{{\rm span}\, x_n\}_{n \in \N}$  is a Riesz basis of subspaces of 
$X$.
\end{remark}
\begin{remark}
If $A$ is a discrete Riesz operator, then clearly $E_n$ commute with $A$ and $A$ is equivalent to  the infinite matrix
\[ 
  A={\mathrm{diag}}( A_1, A_2, \dots, A_n, \dots ),
\]
where $A_n$ is a square matrix which corresponds to the restriction $A|_{ E_nX}$ of $A$.
Then $A= \sum_{n\in \N} i_n A_n E_n$, where $i_n$ is the (natural) inclusion operator and $A_n$ is identified with $(s_n+N_n)|_{E_n}$.
\end{remark}

We consider first order linear port-Hamiltonian systems on a one-dimensional 
spatial domain  
of the form 
\begin{align}
\frac{\partial x}{\partial t}(\zeta,t) &= \left( P_1
  \frac{\partial}{\partial \zeta} + P_0\right) ({\cal H}({\zeta})
x(\zeta,t)),\nonumber \\
x(\zeta,0) &= x_0(\zeta),\label{eqn:pde0}\\
0&= \begin{bmatrix} W_1 & W_0 \end{bmatrix} \begin{bmatrix} ({\cal H}x)(1,t) \\ ({\cal
         H}x)(0,t)\end{bmatrix}, \nonumber
\end{align}
where $\,\, \zeta \in [0,1]$ and $t\ge 0$ and the Assumption \ref{A1} is fulfilled. Equation \eqref{eqn:pde0} describes a special class of port-Hamiltonian systems, which  cover in particular the wave equation, the transport equation and the Timoshenko beam as well as coupled systems. 
For more information  we refer to  \cite{JacZwa12,JacZwa19}. 

\begin{assumption}\label{A1}
The $d\times d$ Hermitian matrix $P_1$ is  invertible, 
$P_0$ is a  $d\times d$  skew-symmetric matrix, $\begin{bmatrix} W_1 & W_0 \end{bmatrix}$ is a full row rank $d\times 
2d$-matrix, and ${\cal H}(\zeta)$ is a positive $d\times d$ Hermitian matrix for a.e.~$\zeta\in (0,1)$  satisfying ${\cal H}, {\cal 
H}^{-1}\in L^{\infty}(0,1;\mathbb C^{d\times d})$. Thus, the matrix $P_1\mathcal{H}(\zeta)$ can be diagonalized as $P_1\mathcal{H}(\zeta)=S^{-1}(\zeta)\Delta(\zeta)S(\zeta)$, where 
$\Delta(\zeta)$ is a diagonal matrix and  $S(\zeta)$ is an invertible matrix  for a.e.~$\zeta\in (0,1)$. We suppose the 
technical assumption that  $S^{-1}$, $S$,  $\Delta: [0,1] \rightarrow \mathbb{C}^{d\times 
d}$ are continuously differentiable. 
\end{assumption}

\begin{definition}
Let $P_0,P_1, \H$ satisfy Assumption \ref{A1} and $X:=L^2((0,1);\mathbb{C}^d)$.
Then the operator $A:\D(A)\subset X\to X$ defined by 
\begin{equation}\label{operatorA}
 Ax:= \left( P_1 \frac{d}{d\zeta} + P_0\right) ({\cal H}x), \qquad x\in {\cal D}(A), 
\end{equation}
\begin{equation}
\label{domainA}
{\cal D}(A) := \left\{ x\in X\mid  {\cal H}x\in H^{1}((0,1);\mathbb C^d) \text{ and } \begin{bmatrix} W_1 & W_0 
\end{bmatrix} \begin{bmatrix} ({\cal H}x)(1) \\ ({\cal
         H}x)(0)\end{bmatrix}=0 \right\},
\end{equation}
is called \emph{port-Hamiltonian operator}.
\end{definition}

We equip the space $X:=L^2((0,1);\mathbb{C}^d)$ with 
the 
energy norm $\sqrt{\langle \cdot,{\cal H} \cdot\rangle}$, where $\langle\cdot,\cdot\rangle$ denotes the standard inner 
product on $L^2((0,1);\mathbb{C}^d)$. Due to our Assumption \ref{A1} the energy norm is equivalent to the standard norm on  
$L^2((0,1);\mathbb{C}^d)$.

Furthermore,  we emphasize that a port-Hamiltonian operator which generate a $C_0$-semigroup is closed and that its
 resolvent is compact, see \cite{Aug16}.

Let $A$ generate  a $C_0$-semigroup $(T(t))_{t\ge 0}$ on $X$. Then 
a closed subspace $V\subset X$ is called \emph{$(T(t))_{t\ge 0}$-invariant}, if $T(t)V\subseteq V$ for all $t\geq 0$.  In this case the restriction $(T(t)|_{V})_{t\ge 0}$ is again a $C_0$-semigroup with generator $A|_{V}$ on $V$, c.f.~\cite[Lemma 2.5.3]{CurZwa95}. 
We say that a set $(s_n)_{n\in\mathbb N}\subset \mathbb C$ has \emph{uniform gap}, if  
$
 \inf_{n\neq m}\abs{s_n-s_m}>0$ for $n,m\in \N$. By $Z^+(\zeta)$, we denote the span of the eigenvectors of $P_1\H(\zeta)$ corresponding to the positive eigenvalues of 
$P_1\H(\zeta)$ and by $Z^-(\zeta)$  the span of the eigenvectors of $P_1\H(\zeta)$ corresponding to the negative eigenvalues 
of 
$P_1\H(\zeta)$.

We are now in  the position to formulate our main result.

\begin{theorem}\label{main}
Let $A$ be a port-Hamiltonian operator and the generator of a $C_0$-semigroup. Then the following is equivalent:
\begin{enumerate}
 \item $A$ is a \riesz.
 \item $A$ is the generator of a $C_0$-group.
 \item $W_1\H(1)Z^+(1)\oplus W_0\H(0)Z^-(0)=W_1\H(1)Z^-(1)\oplus W_0\H(0)Z^+(0)=\C^d$.
\end{enumerate}
If one of the equivalent conditions are satisfied, then $\sigma(A)=\sigma_p(A)$ lie in a  strip parallel to 
the imaginary axis, the eigenvalues (counted according to the algebraic multiplicity) can be decomposed into finitely many sets each having a uniform gap and $A$ satisfies the spectrum determined growth assumption, that is, $\omega_0(A)=s(A)$.
\end{theorem}
The proof of Theorem \ref{main} will be given in Section \ref{secproof}.

\begin{remark}
It is particularly easy to check if a port-Hamiltonian operator $A$ is the generator of a unitary $C_0$-group, c.f.~\cite[Theorem 1.1]{JacMorZwa15}. This is actually the case if and only if
 $\widetilde{W}_B\Sigma \widetilde{W}_B^*= 0$,
where $\widetilde{W}_B:=\begin{bsmallmatrix} W_1 & W_0 \end{bsmallmatrix}\begin{bsmallmatrix} P_1 & -P_1 \\ I & 
I\end{bsmallmatrix}^{-1}
$ and $\Sigma:=\begin{bsmallmatrix}0&I\\ I&0 \end{bsmallmatrix}$.
In this case $A$ is even a skew-adjoint operator by  Stone's Theorem, c.f.~\cite[Theorem VII.3.24]{EngNag00}, which implies that the normalized eigenvectors form an orthonormal basis of $X$.
\end{remark}

\section{Preliminaries on discrete Riesz spectral operators}

Riesz bases of subspace have the following useful characterizations.
\begin{proposition}\cite[Definition 1.4]{GuoZwa01}\label{prop-RBP}
Let $A$ be an operator with compact resolvent and $\sigma(A)=(s_n)_{n\in \mathbb N}$. Then 
the sequence of subspaces $(E_n (X))_{n \in \N}$ is a Riesz basis of subspaces of 
$X$
if and only if there exists positive constants ${m_1}$ and ${m_2}$ such that it holds
\begin{equation*}
{m_1}\norm{x}^2 \leq \sum_{n\in\N} \norm{E_n x }^2\leq {m_2}\norm{x}^2,\quad x\in X.
\end{equation*}
\end{proposition}

Lemma \ref{Riesz-resolvent}, Lemma \ref{pert-resolvent}, Proposition \ref{A=S+N} and Proposition \ref{rieszHans} will be useful for the proof of the main result.
\begin {lemma}\label{Riesz-resolvent}
 Let $A$ be a  \rieszp and $M:=\|N\|$, where $N$ is given by Definition \ref{defDRSO}. Then there exists $C>0$ such that \text{for } $s\in \rho(A)$ with $d(s,\sigma(A))>M$ we have
\begin{equation}\label{riesz-resolvent-ineq}
\|(s-A)^{-1}\|\leq\frac{C}{d(s,\sigma(A))},
\end{equation}
where $d(s,\sigma(A))$ denotes the distance from $s$ to the 
spectrum of $A$.
\end {lemma}
\begin{proof}
Let $\sigma(A)=(s_n)_{n\in \mathbb N}$, $E_n, N_n, N$ as in Definition \ref{defDRSO} 
and $s\in \rho(A)$ with $d(s,\sigma(A))>M$ be arbitrary. 
{ Since $E_n$ is a spectral projection, it commutes with $A$ and the resolvent of $A$. }
By the definition of a discrete spectral operator we have
\[ s-A =\sum_{n=1}^\infty   ((s-s_n)-N_n)E_n\]
and identifying $(s-s_n)-N_n$ with the matrix corresponding to $((s-s_n)-N_n)|_{E_n X}$ we obtain
\[ 
  (s-A)^{-1} =\sum_{n=1}^\infty  ((s-s_n)-N_n)^{-1}E_n.
\]
{ Then it holds for $x \in X$ 
\begin{align*}
 \|(s-A)^{-1}x\|^2&
 =\norms{\sum_{n\in \N}((s-s_n) -N_n)^{-1}E_nx}^2\\ 
&\leq\sum_{n\in \N}\normk{((s-s_n) -N_n)^{-1}E_nx}^2 \\
&
\leq  \sup_{n\in \N} \norm{((s-s_n) -N_n)^{-1}}^2\sum_{n\in \N}\normk{{E}_nx}^2\\ &
\leq {m_2}\sup_{n\in \N}  \norm{((s-s_n) -N_n)^{-1}}^2\norm{x}^2,
\end{align*}}
where $m_2$ is the positive constant of the Riesz basis of subspaces $(E_n)_{n\in \N}$, c.f.~Lemma \ref{prop-RBP}.
Using 
\begin{equation*}
 ((s-s_n)-N_n)=(s-s_n)\left(I-\frac{1}{(s-s_n)}N_n\right)
\end{equation*}
we get
\begin{align*}((s-s_n)-N_n)^{-1}&=\frac{1}{(s-s_n)}\sum_{j=0}^{k_n}\frac{1}{
(s-s_n)^j}N_n^j,
\end{align*}
where $k_n$ denotes the degree of nilpotency of $N_n$.
Thus, for $s\in \rho(A)$ such that $d(s,\sigma(A))>M$, it holds
\begin{align*}\norm{((s-s_n)-N_n)^{-1}}&\leq \sum_{j=1}^{k_n+1} \frac{1}{\abs{(s-s_n)}^j}M^{j-1}\leq  \sum_{j=1}^{k_n+1} 
\frac{1}{d(s,\sigma(A))^j}M^{j-1}\\&\leq  \frac{1}{d(s,\sigma(A))}\sum_{j=0}^{\infty} 
\left(\frac{M}{d(s,\sigma(A))}\right)^j, 
\end{align*}
which concludes the proof.
\end{proof}

\begin{lemma}\label{pert-resolvent}
Let $A$ be a \riesz and generator of a $C_0$-semigroup, and $P\in \BL(X)$. Then there exist constants $K,M>0$ such that for $s \in \rho(A) $ with $d(s,\sigma(A))>M$ we have $s \in \rho(A+P)$ and
\begin{equation*}
\norm{(s-(A+P))^{-1}}\leq \frac{K}{d(s,\sigma(A))}.
\end{equation*}
\end{lemma}

\begin{proof}
By Lemma \ref{Riesz-resolvent}  there exists $M_1,C>0$ such that 
\begin{equation*}
\|(s-A)^{-1}\|\leq\frac{C}{d(s,\sigma(A))},
\end{equation*}
for  $s\in \rho(A)$ with $d(s,\sigma(A))>M_1$. Set $M:=\max\{ M_1,2 \|P\|C\}$. Let $s\in \rho(A)$ with $d(s,\sigma(A))>M$. Then $I-P(s-A)^{-1}$ is invertible and we obtain
\begin{align*}
 \norm{(s-(A+P))^{-1}}&= \|(s-A)^{-1}[I-P(s-A)^{-1}]^{-1}\|\\
&\leq \frac{C}{d(s,\sigma(A))}\frac{1}{1-\norm{P}\norm{(s-A)^{-1}}}\\
&\leq \frac{2C}{d(s,\sigma(A))}.
\end{align*}
\end{proof}

\begin{proposition}\cite{DunSch71}\label{A=S+N}
Let $A$ be an operator with $\sigma(A)=(s_n)_{n\in\N}$ such that the family of spectral subspaces is a Riesz basis of subspaces of $X$. Then $A$ has the representation $A=S+N$,
where the scalar part $S$ is defined as
\begin{align*}
Sx&:=\sum_{n\in \N}s_n E_nx,\\
\D(S)&=\{x\in X\mid \sum_{n=1}^{\infty}\normk{s_n E_nx}^2 < \infty\},
\end{align*} 
and
$
N_n:=N E_n=(A-s_n)E_n$.
Furthermore, $N_n$ is quasi-nilpotent, i.e., $\sigma(N_n)=\{0\}$ for all $n\in \N$.
\end{proposition}
\begin{proposition}\label{rieszHans}
Let $A$ be a generator of $C_0$-group on $X$ with compact resolvent. The eigenvalues are counted with algebraic multiplicity. 
If the following conditions are both fulfilled, 
\begin{enumerate}[label={\Roman*}.]
 \item \label{cond1} The span of the (generalized) eigenvectors form a dense set in $X$; 
 \item \label{cond2} The eigenvalues  can be decomposed into finitely many sets each having a uniform gap;
\end{enumerate}
then $A$ is a \riesz.
\end{proposition}

\begin{proof}
By \cite[Theorem 1.1 and Theorem 1.6]{Zwa10} it follows that 
the family of spectral subspaces is a Riesz basis of subspaces of $X$. 

Thus, $A$ has the  representation $A=S+N$, c.f.\ Proposition \ref{A=S+N}
where $S$ denotes the scalar part of the spectral 
operator $A$ and $N:=\sum_{n\in \N} {N}_n$, where $N_n=(A-s_n){E}_n$ and $N_n$ is quasi-nilpotent. 
To prove that $A$ is a discrete Riesz spectral operator 
it remains to prove hat $N$ is bounded and nilpotent. 

We can identify  $N_n$ with a square matrix  corresponding to $N_n|_{E_nX}$ and thus $N_n$ is bounded and nilpotent.

Since 
the eigenvalues  $A$ can be decomposed into finitely many sets each having a uniform gap, $N$ is nilpotent.

Finally, we verify that $N$ is bounded.
 Without loss of generality we assume that $A$ generates an exponentially stable
$C_0$-group. Then by \cite{LouWex83}, there exists  an invertible and positive operator $L\in \BL(X)$ such that
\begin{equation}\label{Lya}
\<Ax, Lx\>+ \<x, L A x \>=-\<x,x\> \qquad \forall x\in \D(A).
\end{equation}

 We define $A_n:=AE_n=(s_n + N_n)E_n$, where $E_n$ denotes the $n$th spectral projection, and we identify $A_n$ and 
$N_n$ with the corresponding matrices on $E_n X$.
From now on we fix $n$. Then we get for $x\in X$
\begin{align}
\<AE_nx,  LE_nx\> + \<E_nx,L AE_nx\> &=-\<E_nx,E_nx\> \nonumber
\end{align}
or equivalently
\begin{align}  
  \<A_nE_nx, L_nE_nx\> +\<E_nx, L_nA_n E_nx\> &=-\<E_nx,E_nx\>, \label{lyai}
\end{align}
where $L_n:=E_n^* LE_n=L_n^*$. As $L$ is self-adjoint, $L_n$ is  self-adjoint as well. Again we identify $L_n$ with the corresponding matrix on 
$E_n X$ and obtain on $E_n X$
$$  (s_n + {N}_n)^* L_n  + L_n (s_n
	+ {N}_n) = - I.$$
Thus we have 
\begin{equation}\label{eq1}
  {N}_n^* L_n + L_n N_n  =-I+ r_n L_n  
\end{equation}
with $r_n:=-2\Re s_n$. 
Multiplying \eqref{eq1} from the right by $N_n^{k_n-j+1}$ and from the left by $(N_n^*)^{k_n-j}$ with $j=2,3,\ldots,k_n$
gives 
\begin{align*}
 (N_i^*)^{k_n-j+1} L_n N_n^{k_n-j+1} + (N_n^*)^{k_n-j}L_nN_n^{k_n-j+2}& \\
 = -(N_n^*)^{k_n-j} N_n^{k_n-j+1}  &+ r_n ( N_n^*)^{k_n-j} L_n N_n^{k_n-j+1} 
\end{align*}
 and thus it holds
 \begin{align}
\| L_n^{1/2} N_n^{k_n-j+1}\|^2 \leq& \| N_n^{k_n-j}\|\|N_n^{k_n-j+1}\| +  \abs{r_n} \norm{L_n}\|N_n^{k_n-j}\| \|N_n^{k_n-j+1}\|\nonumber\\ &+\|N_n^{k_n-j}\| \|L_n\| \|N_n^{k_n-j+2}\|.\label{eqNi}
\end{align}
Since $L_n$ is boundedly invertible on $E_n X$, we get 
\begin{align}\label{Linv}
m\| N_n^{k_n-j+1} \|^2 \leq \| L_n^{1/2} N_n^{k_n-j+1} \|^2 \text{ for some } m \text{ independent of }n.
\end{align}
For $j=2$ we use $N_n^{k_n}=0$ and 
obtain 
\begin{align*}
m\| N_n^{k_n-1} \| \leq \|N_n^{k_n-2}\| +  \abs{r_n} \norm{L_n}\|N_n^{k_n-2}\|.
\end{align*}
Since $A$ is the generator of a $C_0$-group, we have  $R:=\sup_{n\in \N}{\abs{r_n}}<\infty$. Moreover, with $M:= \sup_{n\in \N}\|L_n\|<\infty$ and $C:=1+RM$, this implies 
 \begin{equation}\label{chainNi}
  \|N_n^{k_n-1} \|\leq \frac{C}{m} \| N_n^{k_n-2}\|.
\end{equation}
For $j=3,\ldots,k_n$ we get using \eqref{Linv}, \eqref{eqNi}, and \eqref{chainNi} and  by induction over $j$
\begin{align}\label{ineqNi}
\|N_n^{k_n-j+1}\| &\leq \|N_n^{k_n-j}\| \sum_{l=1}^{j-1}\frac{C}{m^l}M^{l-1}.
\end{align}
In particular, for $j=k_n$ and using $k_n \leq K$, this implies
\begin{align*}
\|N_n\|\leq \sum_{l=1}^{k_n-1}\frac{C}{m^l}M^{l-1} \le \sum_{l=1}^{K-1}\frac{C}{m^l}M^{l-1}.
\end{align*}

Together with  Proposition \ref{prop-RBP} this implies that $N$ is bounded. 
\end{proof}

 \section{Review on port-Hamiltonian systems}\label{sec:PHS}

The proof of our main result Theorem \ref{main} requires some methods and results from port-Hamiltonian control systems.
We define the \emph{(maximal) port-Hamiltonian operator} by
\begin{equation*}
\frak Ax:= \left( P_1 \frac{d}{d\zeta} + P_0\right) ({\cal H}x), \qquad x\in {\cal D}(\frak A), 
\end{equation*}
 on $X:=L^2((0,1);\mathbb{C}^d)$  with the domain
$
{\cal D}(\frak A) := \left\{ x\in X\mid  {\cal H}x\in H^{1}((0,1);\mathbb C^d) \right\}$.
We equip the linear port-Hamiltonian system \ref{eqn:pde0} with a boundary control $u$, that is, we consider
\begin{align}
\frac{\partial x}{\partial t}(\zeta,t) &= \left( P_1
  \frac{\partial}{\partial \zeta} + P_0\right) ({\cal H}({\zeta})
x(\zeta,t)),\nonumber \\
x(\zeta,0) &= x_0(\zeta),\label{eqn:pde}\\
u(t)&= \begin{bmatrix} W_1 & W_0 \end{bmatrix} \begin{bmatrix} ({\cal H}x)(1,t) \\ ({\cal
         H}x)(0,t)\end{bmatrix}, \nonumber
\end{align}
where $\,\, \zeta \in [0,1]$ and $t\ge 0$ and the Assumption \ref{A1} is fulfilled.

\begin{definition}
 We call the port-Hamiltonian system \eqref{eqn:pde} a
  {\em well-posed control system} if $A$ generates a $C_0$-semigroup on $X$ and
  there exist
  $\tau >0$ and $m_\tau\ge 0$ such that for all $x_0 \in {\cal D}({\mathfrak A})$ and $u \in C^2([0,\tau];\mathbb{C}^d)$ 
with $u(0)= \begin{bsmallmatrix} ({\cal H}x_0)(1,0) \\ ({\cal
         H}x_0)(0,0)\end{bsmallmatrix}$ the classical solution $x$ of \eqref{eqn:pde} satisfies 
   \begin{align*}
    \|x(\tau)\|_{X}^2 \leq
    m_\tau \left( \|x_0\|_{X}^2 + \int_0^{\tau} \|u(t)\|^2 dt  \right).
    \end{align*}
\end{definition}
\vspace{1ex}
There exists a rich literature on well-posed control systems, see e.g.~Staffans \cite{Staf05} and Tuscnak and Weiss 
\cite{TucWei09}. 
In general, it is not easy to show that a control system is well-posed. However, for the   
port-Hamiltonian control system \eqref{eqn:pde}
well-posedness is already satisfied if $A$ generates a $C_0$-semigroup, c.f. 
\cite[Theorem 3.3]{ZwaGorMasVil10} and \cite[Theorem 13.2.2]{JacZwa12}.

In the following, we assume that the 
port-Hamiltonian system \eqref{eqn:pde} is a well-posed control system.
 
Well-posedness implies that for every initial condition $x_0\in X$ and every $L^2$ control function $u$ the port-Hamiltonian control system has a unique  mild solution.

In order to define the mild solution we need to introduce some notation. Let $X_{-1}$ be the completion of $X$ with respect to the norm 
$ \|x\|_{X_{-1}}= \|(\beta  -A)^{-1}x\|_X$ 
for some $\beta $ in the resolvent set $\rho(A)$ of $A$. This implies,
$X\subset  X_{-1}$
and $X$ is continuously embedded and dense in $X_{-1}$. Moreover, let $(T(t))_{t\ge 0}$ be the $C_0$-semigroup generated by $A$ with growth bound $\omega_0(T)$. The $C_0$-semigroup  $(T(t))_{t\ge 0}$ extends uniquely to a $C_0$-semigroup 
$(T_{-1}(t))_{t\ge 0}$ on $X_{-1}$ whose generator $A_{-1}$, with domain equal to $X$, is an extension of $A$, see e.g.\ 
\cite{EngNag00}. 

Well-posedness implies the existence of $\tilde B\in \BL(\C^d,X)$ with $\ran \tilde B\subset\D(\frak A)$ and $\frak A \tilde B \in \BL(\C^d,X)$ such that the unique mild solution of \eqref{eqn:pde} can be defined  by
\begin{equation*}
x(t)=T(t)x_0+ \int_0^t T_{-1}(t-s)(\frak A \widetilde B-A_{-1}\widetilde B) u(s)\, ds,
\end{equation*}
see \cite{JacZwa19}.

The port-Hamiltonian control system can be 
 
written equivalently in the standard control operator formulation
\begin{equation}\label{scf}\dot{x}(t)=A_{-1}x(t)+Bu(t) \quad  x(0)=x_0,\quad t\geq 0,\end{equation}
where
$B\in\BL(\C^d,X_{-1})$ given by \begin{equation}\label{B}B:=\frak A \widetilde 
B-A_{-1}\widetilde B.\end{equation} 
The operator $B$ is an \emph{admissible control operator for $(T(t))_{t\geq 0}$}, that is, 
\begin{equation*}
\Phi_{t_f}u:=\int_0^{t_f}T_{-1}(t_f-s)Bu(s)\;ds\in X
\end{equation*} 
for every $u\in L^2((0,t_f);\C^d)$. Admissibility  implies that the mild solution of \eqref{eqn:pde} satisfies $x\in C([0,t_f);X)$ for every initial condition $x_0\in X$ and every $u\in L^2((0,t_f);\C^d)$.

Well-posed port-Hamiltonian control systems are always exactly controllable in finite time as the following proposition shows.

\begin{proposition}\cite{JacKai19} \label{control}
Every well-posed port-Hamiltonian control system \eqref{eqn:pde} is exactly controllable in finite time, that is, there exists a 
time $\tau >0$ such 
that for all $x_1 \in X$ there
exists a control function $u\in L^2((0,\tau);\mathbb{C}^d)$ such that the corresponding mild solution satisfies $x(0)=0$ 
and 
$x(\tau)=x_1$.
\end{proposition}

As a consequence of exact controllability we obtain that the eigenspaces span the state space.

\begin{proposition}\label{dense}
Consider a well-posed port-Hamiltonian control system \eqref{eqn:pde} and assume that $A$
generates a 
$C_0$-group on $X$. Let $\sigma(A)=(s_n)_{n\in \N}$. Then
\begin{equation*}
 X=\overline{\mathrm{span}_{n\in\N}E((s_n))X},
\end{equation*}
where $E((s_n))$ is introduced in Definition \ref{sp}.
\end{proposition}
\begin{proof}
Follows  from \cite[Lemma 7.3]{JacZwa99} together with Proposition \ref{control}.
\end{proof}


\section{Proof of the Main Result} \label{secproof}

This section is devoted to the proof of Theorem  \ref{main}.
Let us first assume that one of the conditions of the theorem is satisfied. As the resolvent of $A$ is compact the spectrum consists of isolated eigenvalues only. That the eigenvalues lie in a strip parallel to the imaginary axis and that they can be decomposed into finitely many sets each having a uniform gap will be shown in the proof of the implication $2)\Rightarrow 1)$.
Finally, $\omega_0(A)= s(A)$ is implied by \cite[Theorem 2.12]{GuoZwa01}.

\subsection{Proof of the equivalence $2)\Leftrightarrow 3)$ of Theorem \ref{main}}

The operator $A$ generates a $C_0$-group, if and only if $A$ and $-A$ generates a $C_0$-semigroup, \cite[Section II.3.11]{EngNag00}. In \cite[Theorem 1.5]{JacMorZwa15} shows that $A$ is the generator of a 
$C_0$-semigroup if and only if $W_1\H(1)Z^+(1)\oplus W_0\H(0)Z^-(0)=\C^d$. Since $\D(-A)=\D(A)$ and $\sigma_p(P_1\H)=
- \sigma_p(-P_1\H)$ it follows that $-A$ generates a 
$C_0$-semigroup if and only 
if $W_1\H(1)Z^-(1)\oplus W_0\H(0)Z^+(0)=\C^d$.

\subsection{Proof of the implication $2)\Rightarrow 1)$ of Theorem \ref{main}}

The following lemma will be useful.
\begin{lemma}\label{lsgode}
Let $s\in \C$ and $P_1,P_0$ and $\H$ fulfil the condition for a port-Hamil\-to\-ni\-an operator in Assumption \ref{A1}. Then the  solutions of the system of ordinary differential equations
\begin{equation}\label{ode}
 sx(\zeta)= \left(P_1\frac{d}{d \zeta}+P_0 \right)(\H x)(\zeta), \quad \zeta \in [0,1],
\end{equation}
denoted by $x(\zeta)=\Psi^s(\zeta)x(0)$, satisfy
$$  \tilde M e^{-\abs{\Re s}\tilde c_0\zeta}\norm{v}\leq\norm{\Psi^s(\zeta)v} \leq M e^{\abs{\Re 
s}c_0\zeta}\norm{v},\quad v\in \C^d,\; \zeta \in [0,1],$$
with constants $M,\tilde M>0$, and $\tilde c_0, c_0\geq 0$ independent of $s$ and $\zeta$.
\end{lemma}

\begin{proof}
 Writing $\tilde x=\H x$, \eqref{ode} can be equivalently written as 
 \begin{align*}
  \H(\zeta) P_1 \tilde x'(\zeta)=s\tilde x(\zeta)-\H(\zeta) P_0\tilde x(\zeta).
 \end{align*}

We write $s=i\omega+r$ with $\omega, \tau \in \R$ and diagonalize $P_1 \H(\zeta)=S^{-1}(\zeta)\Delta(\zeta)S(\zeta)$, see Assumption \ref{A1}. Thus,
$
\H(\zeta) P_1
=S^{*}(\zeta)\Delta(\zeta)S^{-*}(\zeta)
$
and we get
\begin{align*}
& \Delta(\zeta) S^{-*}(\zeta) \tilde x'(\zeta)=i\omega S^{-*}(\zeta) \tilde 
x(\zeta)+ \left(r I-S^{-*}(\zeta) \H(\zeta) P_0 
S^{*}(\zeta)\right)S^{-*}(\zeta)\tilde x(\zeta).
\end{align*}
 Using the substitution 
$z=S^{-*}\tilde x$ gives the equivalent differential equation
\begin{align}
 z'(\zeta)&=i\omega \Delta^{-1}(\zeta)z(\zeta)+ \left(r \Delta^{-1}(\zeta)+Q(\zeta) 
\right)z(\zeta), 
\label{z}
\end{align}
where $$Q(\zeta):=-\Delta^{-1}(\zeta)S^{-*}(\zeta)\H(\zeta) 
P_0 S^{*}(\zeta)- (S^{-*})'(\zeta)S^{*}(\zeta).$$
Thus, equation \eqref{ode} is equivalent to equation \eqref{z}.
Due to the fact that $P_1\H(\zeta)$ has real eigenvalues, $\Delta(\zeta)$ is a diagonal, real matrix and
$i\omega\Delta^{-1}(\zeta)$ is a diagonal, purely imaginary matrix. 
We write $\Delta^{-1}(\zeta)=\diag_{k=1,\ldots,n}(\alpha_k(\zeta))$ with $\alpha_k(\zeta):[0,1]\to \R$ and  define $\Phi_{\omega}(\zeta)=\diag(\exp(-i\omega\int_0^{\zeta}\alpha_k(\tau)d\tau))$ which satisfies
$$\norm{\Phi_{\omega}(\zeta)}_{\BL(\C^d)} 
=1, \qquad \zeta\in[0,1].$$
Multiplying \eqref{z} 
with $\Phi_{\omega}(\zeta)$, we get
\begin{align*}
&\Phi_{\omega}(\zeta) z'(\zeta)-i\omega 
\Delta^{-1}(\zeta)\Phi_{\omega}(\zeta)z(\zeta)=\Phi_{\omega}(\zeta)\left(r \Delta^{-1}(\zeta)+Q(\zeta) \right)z(\zeta)
\end{align*}
or equivalently 
\begin{align*}
\left(\Phi_{\omega}(\zeta)z(\zeta)\right)'=
 \left(r\Delta^{-1}(\zeta)+\Phi_{\omega}(\zeta)
Q(\zeta)\Phi_{\omega}^{-1}(\zeta)\right)\Phi_{\omega}(\zeta)z(\zeta).
\end{align*}
Using the substitution $y=\Phi_{\omega}z$, this ordinary differential equation becomes
\begin{align}\label{yode}
y'=(r \Delta^{-1}(\zeta)+ Q_{\omega}(\zeta) ) y(\zeta),
\end{align}
where $ Q_{\omega}(\zeta):=\Phi_{\omega}(\zeta)Q(\zeta)\Phi_{\omega}^{-1}(\zeta)$.
There exist constants $c_0, c_1 \geq 0$,  independent of $\omega$, such that
\begin{equation}\label{gr2}
2 \max_{\zeta \in [0,1]}\|r\Delta^{-1}(\zeta)+ Q_{\omega}(\zeta)\| \leq \abs{r} c_0+c_1.
\end{equation}
The solution $y$ of \eqref{yode} satisfies
\begin{align}\label{eqn:ode1}
\frac{d}{d\zeta}\|y(\zeta)\|^2=& y(\zeta)^*[(r \Delta^{-1}(\zeta)+ Q_{\omega}(\zeta) ) +(r \Delta^{-1}(\zeta)+ Q_{\omega}(\zeta) )^*]y(\zeta).
\end{align}
This together with \eqref{gr2} implies
\begin{equation*}
-(\abs{r} c_0+c_1)\|y(\zeta)\|^2\le \frac{d}{d\zeta}\|y(\zeta)\|^2\le (\abs{r} c_0+c_1)\|y(\zeta)\|^2,
\end{equation*}
and thus 
\begin{align*}\label{gronwall}
e^{- (\abs{r} c_0+c_1)\zeta}\norm{y(0)}^2 \le \norm{y(\zeta)}^2\leq 
e^{ (\abs{r} c_0+c_1)\zeta}\norm{y(0)}^2.
\end{align*}
As the mapping $x\mapsto y$ is boundedly invertible on $
 L^2((0,1),\mathbb C^d)$, with norm independent on $\omega$, the statement follows.
\end{proof}
%

Next we state some results from complex analysis.

\begin{definition} \cite[II.1.27]{AvdIva95}
 An entire function $f$ is called an \emph{entire function of exponential type}, if there exist constants $C$ and $T$ such 
that $\abs{f(s)}\leq C e^{T\abs{s}} $ for all $s \in \C$.
 Further, an entire function $f$ of exponential type is said to be of \emph{sine 
type}, if
\begin{enumerate}
 \item the zeros of $f$ lie in a strip $\{s \in \C \mid \abs{\Im s}\leq h\}$ for some $h\geq 0$.
 \item there exist $\hat \omega \in \R$ and positive constants $c$ and $C$ such that $c\leq \abs{f(r+i\hat \omega)}\leq C$ 
for every $r \in \R$ holds.
\end{enumerate}
\end{definition}

\begin{proposition}\label{stf}\cite[Proposition II.1.28]{AvdIva95} (Levin 1961) If f is of sine type, then its set of 
zeros counted with algebraic multiplicity is a finite unification of 
sets each having a uniform gap.
\end{proposition}

\begin{lemma}\label{deteig}
A complex number $s\in \C$ is an eigenvalue of a port-Hamiltonian operator $A$ if and only if
\begin{equation*}\det \begin{bmatrix} W_1\H(1)\Psi^s(1)+W_0\H(0) \end{bmatrix}=0,
\end{equation*}
where $\Psi^s$ is described in Lemma \ref{lsgode}.
\end{lemma}

\begin{proof}
For every $x(0)\in \C^d$ there exists a solution of the differential equation $sx(\zeta)= (P_1\frac{d}{d \zeta}+P_0 )(\H x)(\zeta)$, $ \zeta \in [0,1]$.
The complex number $s$ is an eigenvalue of $A$ if and only if $x\in \D(A)$ and $Ax=sx$. 
Using Lemma \ref{lsgode} this is equivalent to
\begin{align*}
x(\zeta)=\Psi^s(\zeta)x(0) \text{ and }\begin{bmatrix} W_1& W_0\end{bmatrix}\begin{bmatrix} (\H x)(1)\\ (\H x)(0)\end{bmatrix}=0.
\end{align*}
Inserting the first equation in the second, we get that $s$ is an eigenvalue of $A$ if and only if
$ \det\begin{bmatrix} W_1\H(1)\Psi^{s}(1)+W_0\H(0)\end{bmatrix}=0$. 
\end{proof}


Now we are in the situation to give the proof of the implication $2)\Rightarrow 1)$.

\begin{proof}[ Proof of the implication $2)\Rightarrow 1)$ of Theorem \ref{main}]
Assertion $2)$ implies that the eigenvalues lie in a strip parallel to the imaginary axis. Thus, thanks to Proposition \ref{dense} and Proposition \ref{rieszHans} it suffices to show that the eigenvalues $(s_n)_{n\in \N}$ of $A$ (counted according to the algebraic multiplicity) can decomposed into finitely many sets having each a uniform gap.
 
Using Proposition \ref{stf}, this is implied by  the existence of an 
entire function $g$ of exponential type with 
\begin{enumerate}[label=\roman*),itemsep=0.2ex]
	\item $g$ has exactly the zeros $(s_n)_{n\in \N}$ and 
	\item there exist $r,c,C>0$ such that for every $\omega\in \R$: $c\leq 
\abs{g(r+i\omega)}\leq C$.
\end{enumerate}
Note that $f:\C\to \C$, defined by $f(s):=g(is)$ is a sine-type function if $g$ is an entire function of exponential type and satisfying the above conditions $i)$ and $ii)$.
We define $g:\C\to \C$ by \begin{align}
 g(s):=\det\begin{bmatrix} W_1\H(1)\Psi^{s}(1)+W_0\H(0)\end{bmatrix}. \label{stfg}
\end{align} 
By Lemma \ref{deteig}, a complex number $s\in \C$ is an eigenvalue of the operator $A$ if and 
only if $g(s)=0$.

$\Psi^s$ described in Lemma \ref{lsgode} is an entire function, c.f. \cite[Theorem 24.1]{Was87} and thus $g$ is an entire function as well. Clearly, $g$ has the zeros $(s_n)_{n\in \N}$. 
Since the determinant of a matrix equals the product of its eigenvalues and 
every eigenvalue is smaller or equal the norm of the matrix, it yields $\abs{g(s)}
\leq 
\|\begin{bmatrix} W_1\H(1)\Psi^{s}(1)+W_0\H(0)\end{bmatrix}\|^n$.
Using Lemma \ref{lsgode} it holds {$\abs{g(s)}\leq 
c_2 e^{\abs{\Re s}c_3}$} for some constants $c_2,c_3 \geq 0,
$
and thus, $g$ is bounded on lines parallel to the imaginary axis and grows at most exponentially.

Next, we show  that $g$ is bounded away from zero on some line parallel to the imaginary axis.
Since the control operator $B$ of the port-Hamiltonian system \eqref{eqn:pde} is admissible, see Section \ref{sec:PHS}, it yields that for
$\omega>\omega_0(A)$ exists a constant $M_{\omega}>0$ such that
\begin{equation}\label{UB}
\|(s-A_{-1})^{-1}B\|_{\BL(\C^d, X)}\leq \frac{M_{\omega}}{\sqrt{\Re s-\omega}} \quad \text{for } \Re 
s\geq\omega,
\end{equation}
see  \cite[Proposition 
4.4.6]{TucsWeis14}. 
Let $r>\omega_0(A)$ and we assume that $g$ is not bounded away from zero on $r+i\R$, i.e., there exists a sequence 
$\omega_k 
\in \R$ such that $g(r+i\omega_k)\to 0$. Since all zeros of $g$ have real part less or equal to the growth bound 
$\omega_0(T)$ of the 
$C_0$-semigroup generated by $A$, it holds true that $g(r+i\omega_k)\neq 0$. 
Let 
$u_0$ be an arbitrary vector in 
$\C^d$. By \cite[Theorem 12.1.3]{JacZwa12} the solution $x_{u_0}^{r+i\omega_k}$ of 
 \begin{align*}\label{bcsu0}
 (r+i\omega_k)x(\zeta)= \left(P_1\frac{d}{d \zeta}+P_0 \right)(\H x)(\zeta)\\
 \begin{bmatrix}W_1 & W_0 \end{bmatrix} \begin{bmatrix} (\H x)(1) \\(\H x)(0)\end{bmatrix}=u_0,
 \end{align*}
 is given by
\begin{align*}
x_{u_0}^{r+i\omega_k}&=(((r+i\omega_k)-A_{-1})^{-1}(\frak A \tilde B-(r+i\omega_k)\tilde B)+\tilde B)u_0\\
&
=((r+i\omega_k)-A_{-1})^{-1}Bu_0
\end{align*}
and $x_{u_0}^{r+i\omega_k}(0)$ fulfils 
\begin{align*}
\begin{bmatrix} W_1\H(1)\Psi^{r+i\omega_k}(1)+W_0\H(0)\end{bmatrix}x_{u_0}^{r+i\omega_k}(0)=u_0.
\end{align*} 

Let $M_k:=\begin{bmatrix} W_1\H(1)\Psi^{r+i\omega_k}(1)+W_0\H(0)\end{bmatrix}$.
Since $g(r+i\omega_k)=\det(M_k)\to 0$ and $1=\det(I)
=\det(M_k)\det(M_k^{-1})$, we have 
$\det(M_k^{-1})\to \infty$. 
Hence, $M_k^{-1}$ has an eigenvalue $\nu_k$ with $\abs{\nu_k}\to \infty$. Choose $u_{k,0}$ as a normalized 
eigenvector to $\nu_k$. Then it yields 
\begin{align*}
x^{r+i\omega_k}_{u_{k,0}}(0)=M_k^{-1}u_{k,0}=\nu_k u_{k,0}, 
\end{align*} 
which implies
$\|x^{r+i\omega_k}_{u_{k,0}}(0)\|_{\C^d}\to \infty$. 
We note that the function $x^{r+i\omega_k}_{u_{k,0}}$ is given by 
$$x^{r+i\omega_k}_{u_{k,0}}(\zeta)=\Psi^{r+i\omega_k}(\zeta)x^{r+i\omega_k}_{u_{k,0}}(0).
$$
Using that 
the inverse of $\Psi^{r+i\omega_k}$ is a bounded function, see Lemma \ref{lsgode}, we get
\begin{equation}\label{unbounded}\|x^{r+i\omega_k}_{u_{k,0}}\|_{L^2((0,1);\C^d)}\to \infty.\end{equation} 
However, since we also have
$x^{r+i\omega_k}_{u_{k,0}}=((r+i\omega_k)-A_{-1})^{-1}Bu_{k,0}$, equation \eqref{unbounded} is in contradiction with the uniform boundedness of 
$((r+i\omega_k)-A_{-1})^{-1}Bu_{0,k}$, see equation \eqref{UB}.
Thus the entire function $g$ is of exponential type and satisfies condition $i)$ and $ii)$. This concludes the proof.
\end{proof}

\subsection{Proof of the implication $1)\Rightarrow 2)$ of Theorem \ref{main}}

We start with some characterizations of the resolvent and the spectrum of port-Hamiltonian operators.

\begin{lemma}\label{resolvent-explode}
Let  $\Lambda\in C([0,1];\C^d)$  with $\Lambda(\zeta)$   diagonal, invertible und positive definite for every $\zeta\in[0,1]$,  $Q \in \C^{n\times n}$  singular  and  $A:\D(A) \subset L^2((0,1));\mathbb C^d)\rightarrow L^2((0,1));\mathbb C^d)$ defined by $Ax=\Lambda x' $ and  $\D(A)=\{x\in H^1((0,1);\C^d) \mid x(1)+Qx(0)=0 \}$. Then there exist real constants $\gamma< 0$ and $a,b>0$ such that
$\norm{(s-A)^{-1}}\geq a e^{b\abs{s}}$ for real $s\in \rho(A)$ with $s < \gamma$.
\end{lemma}

\begin{proof}
Let   $0\neq x(0)\in \ker Q$. We define 
$F_s(\zeta):=s\int_0^{\zeta}\Lambda^{-1}(\tau)d\tau $ for $s\in \mathbb R$ and $\zeta\in [0,1]$. Note that $F_s$ and $\Lambda$ commute as both are diagonal.
There exists $\gamma_0<0$ such that $I-e^{2F_s(1)}$ is invertible if $\, s<\gamma_0$.
Further, let $s\in \rho(A)$ with $\, s<\gamma_0$
 and define
\begin{equation}g(\zeta)=2s  e^{F_s(1)-F_s(\zeta)} g_0=2s e^{s\int_{\zeta}^1 \Lambda^{-1}(\tau)d\tau}g_0 \label{defg}
\end{equation}
with $g_0:=e^{F_s(1)} ( I-e^{2F_s(1)} )^{-1}x(0)$.
Then the solution of
\begin{equation}\label{eq:res}\Lambda x'=sx+g 
\end{equation} is given by
\begin{align*}
x(\zeta)& 
=e^{ F_s(\zeta)}x(0)+\int_0^{\zeta} e^{F_s(\zeta)-F_s(\tau)} \Lambda^{-1}(\tau)g(\tau)  d \tau \\
&=e^{ F_s(\zeta)}x(0)-   e^{F_s(\zeta)+F_s(1)}\int_0^{\zeta}  (-2s\Lambda^{-1}(\tau)  ) e^{-2F_s(\tau)}d \tau   \,  g_0\\
&=e^{ F_s(\zeta)}x(0)-   e^{F_s(\zeta)+F_s(1)}(e^{-2F_s(\zeta)}  -I ) g_0 \\
&=e^{ F_s(\zeta)}(I-e^{2 F_s(1)})e^{-F_s(\zeta)}g_0-e^{-F_s(\zeta)+F_s(1)}g_0+e^{-F_s(\zeta)+F_s(1)}g_0\\
&= e^{F_s(\zeta)-F_s(1)} g_0   - \frac{1}{2s}  g(\zeta).
\end{align*}
In particular $x(1)=0$ and thus $(s-A)^{-1}g=x$. As $\Lambda\in C([0,1];\C^d)$  is a diagonal and invertible matrix-valued function with positive entries on its diagonal, there exists $\lambda_0>0$ such that the diagonal elements of $\Lambda$ are bounded  by $\lambda_0$.
First using \eqref{defg} we can estimate 
$$\norm{g}_{L^2((0,1);\C^d)} \leq c \sqrt{|s|}\norm{g_0}, \; c>0,$$
and 

\begin{align*}
\norm{x}_{L^2(0,1;\C^d)}&\geq \norm{e^{F_s(\cdot)-F_s(1)}g_0}_{L^2(0,1;\C^d)}-\frac{1}{2|s|}\norm{g}_{L^2((0,1);\C^d)} \\
&= \left(\int_0^1 \norm{e^{-s\int_\tau^1\Lambda^{-1}(\sigma) d\sigma}g_0}^2 d\tau \right)^{\frac{1}{2}}-\frac{1}{2|s|}\norm{g}_{L^2((0,1);\C^d)} \\
&\geq \left(\int_0^1 e^{2\abs{s}(1-\tau)\lambda_0}\norm{g_0}^2 d\tau \right)^{\frac{1}{2}}-\frac{1}{2|s|}\norm{g}_{L^2((0,1);\C^d)}\\
&\geq \frac{1}{4|s|^2\lambda_0}\left( e^{2\abs{s} \lambda_0}-1\right) \norm{g}_{L^2((0,1);\C^d)}  -\frac{1}{2|s|}\norm{g}_{L^2((0,1);\C^d)}.
\end{align*}
This completes the lemma.
\end{proof}

\begin{lemma}\label{large-gap}
Let $A$ be a port-Hamiltonian operator, which generates a $C_0$-semi\-group. Furthermore, let $A$ be a \riesz and let $(s_n)_{n\in\N}$ denote its eigenvalues. Then there exists a constant $K>0$ such that for every $n\in \N$ within the ball $\{s\in \C\mid \abs{s-s_n}\leq K\abs{\Re s_n}^2\}$ there lie at most $d$ eigenvalues.

\end{lemma}

\begin{proof}
Without lost of generality we assume that $A$ generates an exponentially stable $C_0$-semigroup. 
By Proposition \ref{control} the corresponding port-Hamiltonian control system \eqref{eqn:pde} with control operator $B$ 
is 
exactly 
controllable in finite time. 
Then the dual system, described by $A^*$ and $B^*$, is exactly observable and it yields due to the Hautus Test, c.f. \cite{RusWei94}, 
that there exists a positive constant $m$ such that
\begin{equation}\label{Hautus}
 \norm{(s-A^*)x}^2+\abs{\Re s}\norm{B^*x}^2\geq m \abs{\Re s}^2\norm{x}^2, \quad \Re s <0, x\in \D(A^*).
\end{equation}
No matter that there may exist generalized eigenvectors, we consider only eigenvectors corresponding to different 
eigenvalues of $A^*$. As $\sigma(A)=\overline{\sigma(A^*)}$, it suffices to prove the statement for $A^*$.
Choose arbitrary $e_1,\ldots e_{d+1}$ normed eigenvectors of the operator $A^*$ to the eigenvalues $\lambda_1,\ldots, \lambda_{d+1}$ with $\lambda_n\neq \lambda_{m}$  for  $n\not=m$ $\in 1, \ldots, d+1$.

Since $B^*\in \BL(\D(A^*),\C^d)$ the $d+1$ vectors $B^*e_n$ are linearly dependent in $\mathbb C^d$, i.e,
 there exists scalars $a_1,\ldots,a_{d+1} \in \C$ with 
$\sum_{n=1}^{d+1}\abs{a_n}^2=1$ such that
\begin{equation}
 a_1B^*e_1+\ldots a_{d+1}B^*e_{d+1}=0.
\end{equation}
Consider $x=\sum_{n=1}^{d+1}a_ne_n$. Then $x\in \D(A^*)$ with 
$B^*x=0$, and $$\norm{x}^2=\norms{\sum_{n=1}^{d+1}a_ne_n}\ge m_1>0,$$ by Proposition \ref{prop-RBP}.
It follows with the Hautus Test \eqref{Hautus} at the point $s=\lambda_{d+1}$ 
\begin{align*}
m \,m_1 \abs{\Re \lambda_{d+1}}^2&\leq  
\norm{(\lambda_{d+1}-A^*)x}^2
=\norms{\sum_{n=1}^{d}a_n(\lambda_{d+1}-\lambda_n)e_n}^2\\&\leq 
m_2\sum_{n=1}^{d}\abs{a_n} ^2\abs{
\lambda_{d+1}-\lambda_n } ^2.
\end{align*}
Thus,\begin{align*}
\frac{d+1}{d+1}\frac{m \,m_1}{m_2} \abs{\Re \lambda_{d+1}}^2&\leq  
\sum_{n=1}^{d}\abs{
\lambda_{d+1}-\lambda_n }^2.
\end{align*}
Since the eigenvalues are arbitrary chosen, this implies that in the ball with radius $\frac{m \,m_1}{(d+1)\; m_2}\abs{\Re \lambda_{d+1}}^2 $ around $\lambda_{d+1}$ lies at most $d$ eigenvalues.
Hence, the statement follows.
\end{proof}

Now we are in the position to prove the implication $1)\Rightarrow 2)$.
%
\begin{proof} [Proof of the implication $1)\Rightarrow 2)$ of Theorem \ref{main}:]
We assume that $A$ does not generate a $C_0$-group. 

Since $A$ is a port-Hamiltonian operator, we denote by $S$ the matrix-valued function such that $P_1\H(\zeta)=S^{-1}(\zeta)\Delta(\zeta)S(\zeta)$, see Assumption \ref{A1}.
Since the eigenvalues of $P_1\H(\zeta)$ and $\H(\zeta)^{\frac{1}{2}}P_1\H(\zeta)^{\frac{1}{2}}$ are the same,  it follows by Sylvester’s law of inertia that the number of positive and negative eigenvalues of $P_1\H(\zeta)$ equal those of $P_1$.
Let $d_1$ denote the number of positive and $d_2=d-d_1$ the number of negative eigenvalues of $P_1$.
Thus without loss of generality $\Delta$ can be written as $\Delta(\zeta)=\begin{bsmallmatrix}\Lambda(\zeta)& 0\\ 0&\Theta(\zeta)\end{bsmallmatrix}$, where $\Lambda(\zeta)\in \C^{d_1\times d_1}$ correspond to the  positive eigenvalues  and $\Theta(\zeta)\in \C^{d_2\times d_2}$ to the negative ones.

 Let $\mS\in \BL(X)$ be the multiplication operator $(\mS x)(\zeta):= S(\zeta)x(\zeta)$. By assumption $\mS$ is invertible and we obtain
 \begin{align*}
 \mS A \mS^{-1}z(\zeta)=&\Delta(\zeta)(z(\zeta))'\\
 & +\Delta'(\zeta)z(\zeta) +S(\zeta)(S^{-1})'
 \Delta(\zeta)z(\zeta)+S(\zeta)P_0\H(\zeta)S^{-1}(\zeta)z(\zeta)\\
 \D(\mS A\mS^{-1})=&\{z\in H^1((0,1);\C^d)\mid \begin{bmatrix} W_1 & W_0 
\end{bmatrix} \begin{bsmallmatrix} (\H S^{-1}z)(1)\\(\H S^{-1}z)(0) 
  \end{bsmallmatrix}
=0\}.
\end{align*}
The operator $ \mS A \mS^{-1}$ generates a $C_0$-semigroup, too.
We split the variable $z(\zeta)=\begin{bsmallmatrix}z^+(\zeta)\\ z^-(\zeta)\end{bsmallmatrix}\in \C^{d}$ with $z^+(\zeta)\in \C^{d_1}$ and $z^-(\zeta)\in \C^{d_2}$, and
we define $W_1\H(1) S^{-1}(1)=:\begin{bmatrix}V_1 & V_2\end{bmatrix}$ and $W_0\H(0) S^{-1}(0)=:\begin{bmatrix}U_1 & U_2\end{bmatrix}$, where $U_1,V_1\in\C^{d\times d_1}$ and $U_2,V_2\in \C^{d\times d_2}$. 
Then, $z\in \D( \mS A \mS^{-1})$ if and only if $z\in H^1((0,1);\C^d)$ and
\begin{align*}
0&= \begin{bmatrix} W_1& W_0\end{bmatrix} \begin{bmatrix} (\H S^{-1}z)(1)\\ (\H   S^{-1}z)(0)\end{bmatrix}\\&=\begin{bmatrix}V_1 & V_2\end{bmatrix}\begin{bmatrix} z^+(1)\\ z^-(1)\end{bmatrix}+\begin{bmatrix}U_1 & U_2\end{bmatrix}\begin{bmatrix} z^+(0)\\ z^-(0)\end{bmatrix}\\&=\begin{bmatrix}V_1 & U_2\end{bmatrix}\begin{bmatrix} z^+(1)\\ z^-(0)\end{bmatrix}+\begin{bmatrix}U_1 & V_2\end{bmatrix}\begin{bmatrix} z^+(0)\\ z^-(1)\end{bmatrix}\\&=K\begin{bmatrix}\Lambda(1) z^+(1)\\ \Theta(0)z^-(0)\end{bmatrix}+Q\begin{bmatrix}\Lambda(0) z^+(0)\\ \Theta(1)z^-(1)\end{bmatrix},
\end{align*}
where $K:=\begin{bmatrix}V_1 & U_2\end{bmatrix}\begin{bsmallmatrix} \Lambda(1)^{-1}& 0\\0&\Theta(0)^{-1}\end{bsmallmatrix}$ and $Q:=\begin{bmatrix}U_1 & V_2\end{bmatrix}\begin{bsmallmatrix} \Lambda(0)^{-1}& 0\\0&\Theta(1)^{-1}\end{bsmallmatrix}$.
Since $ \mS A \mS^{-1}$ is the generator of a $C_0$-semigroup and this property is invariant under bounded perturbations, $K$ is invertible, see \cite[Theorem 13.3.1]{JacZwa12}. 

Let $\T\in \BL(X)$ be defined by $\T  \begin{bsmallmatrix}z^+(\zeta)\\ z^-(\zeta)\end{bsmallmatrix}:= \begin{bsmallmatrix}z^+(\zeta)\\ z^-(1-\zeta)\end{bsmallmatrix}$. Clearly, $\T$ is invertible.
Then the operator $\mA:=\T\mS A \mS^{-1}\T^{-1}$ on $X$ is given by
 \begin{align*}
 \mA z(\zeta)&=\begin{bmatrix}\Lambda(\zeta)& 0\\ 0&-\Theta(1-\zeta)\end{bmatrix}z'(\zeta) +
 R(\zeta)z(\zeta)\\
 \D(\mA)&=\left\{z\in H^1((0,1);\C^d)\mid K\begin{bmatrix}\Lambda(1) &0\\ 0 &\Theta(0)\end{bmatrix}z(1)+Q\begin{bmatrix}\Lambda(1) &0\\ 0 &\Theta(0)\end{bmatrix}z(0)
=0\right\},
\end{align*}
 where $z\mapsto R(\cdot)z(\cdot)$ is a bounded multiplication operator on $X$.
 Let
 $\tilde K:= K\begin{bsmallmatrix}\Lambda(1) &0\\ 0 &\Theta(0)\end{bsmallmatrix}$ and $\tilde Q:= Q\begin{bsmallmatrix}\Lambda(1) &0\\ 0 &\Theta(0)\end{bsmallmatrix}$. By assumption, the matrix $\tilde K$ is invertible.
As a bounded perturbation of a generator of $C_0$-group generates again a $C_0$-group, we obtain that the operator
\begin{align*}
 \tilde {A} z(\zeta)
&=\begin{bmatrix}\Lambda(\zeta)& 0\\ 0&-\Theta(1-\zeta)\end{bmatrix} z'(\zeta)\\
 \D(\tilde A)&=\left(z\in H^1((0,1);\C^d)\mid z(1)=\tilde K^{-1}\tilde Q z(0) \right\}
\end{align*}
 generates a $C_0$-semigroup, but not a $C_0$-group. In particular, $Q_1:=\tilde K^{-1}\tilde Q$ is singular.

Since $A$ is a \riesz, due to Lemma \ref{large-gap}, there is a $K>0$ such that in the ball with radius $K\abs{\Re s}^2$ around every eigenvalue $s$ of $A$ lie at most $d$ eigenvalues. Thus there exist sequences $(t_n)_{n\in\N}\subset \mathbb R$  and $(r_n)_{n\in\N}\subset (0,\infty)$ with $t_n\rightarrow -\infty$ and $r_n\rightarrow \infty$ such that the ball with center $t_n$ and radius $r_n$ lie in $\rho(A)$. By Lemma \ref{pert-resolvent} we get
$$ \|(t_n-\tilde A)^{-1}\|\rightarrow 0.$$
However, for the operator $\tilde A$ the Lemma \ref{resolvent-explode}, is applicable which implies that  $\|(t_n-\tilde A)^{-1}\|\rightarrow \infty$.
This leads to a contradiction.   
\end{proof}

\section{Examples}

We consider  the one-dimensional wave equation  with boundary feedback as in  \cite{XuWei11}, but we allow for spatial dependent mass density and Young's modulus, given by
\begin{align}
w_{tt}(\zeta,t)&=\frac{1}{\rho(\zeta)} \frac{\partial }{\partial \zeta }(T(\zeta)w_{\zeta}(\zeta,t)), \quad \zeta \in [0,1], t\geq 0, \nonumber\\
w(0,t)&=0\nonumber\\
u(t)&=T(1)w_{\zeta}(1,t)\label{cls}\\
y(t)&=w_{t}(1,t),\nonumber\\
u(t)&=-\kappa y(t), \kappa >0,\nonumber
\end{align}
where $\zeta\in [0,1]$ is the spatial variable, $w(\zeta,t)$ describes the displacement of the point $\zeta$ of the string at time $t$, $T (\zeta)>0 $ is the Young's modulus of the string,  $\rho (\zeta ) >0$ is the mass density, and $\kappa>0$. 
We model this system as a port-Hamiltonian system. Therefore we introduce the state variable
 $x=\begin{bsmallmatrix} x_1(\zeta,t) \\ x_2(\zeta,t) \end{bsmallmatrix}$ with 
$x_1=\rho(\zeta)\frac{\partial w}{\partial t}$ (momentum), $x_2 = \frac{\partial w}{\partial \zeta}$ (strain)  and the state 
space $X=L^2((0,1);\C^d)$.
Since the meaning of $w(0,t)=0$ is that in the point $\zeta=0$ the string is fixed for all times, we model this boundary condition as $\frac{\partial w}{\partial t}w(0,t)=0$.
 Then the closed-loop system \eqref{cls} can equivalently be written as
\begin{align}
  \nonumber
  \frac{\partial }{\partial t} \begin{bmatrix} x_1(\zeta,t) \\ x_2(\zeta,t) \end{bmatrix} 
 &= P_1  \frac{\partial }{\partial \zeta}\left( {\cal H}(\zeta)\begin{bmatrix} x_1(\zeta,t) \\ x_2(\zeta,t) 
\end{bmatrix} \right),\\
0&=\begin{bmatrix}
0 & 0 & 1&0 \\
\kappa & 1 & 0 &0
\end{bmatrix}
\begin{bmatrix}
\H(1) x(1,t)\\
\H(0) x(0,t)
\end{bmatrix},
\end{align}
where 
$ P_1=\begin{bsmallmatrix} 0 & 0\\ 0 & 1 \end{bsmallmatrix},  {\cal 
H}(\zeta)=\begin{bsmallmatrix}\frac{1}{\rho(\zeta)} & 0\\ 0 & T(\zeta)  \end{bsmallmatrix}, W_1=\begin{bsmallmatrix}
0 & 0\\
\kappa & 1 
\end{bsmallmatrix}, W_0=\begin{bsmallmatrix}
1 & 0\\
0 &0
\end{bsmallmatrix}
$ and $\kappa >0$.
We define the corresponding port-Hamiltonian operator $A$

\begin{align*}
Ax&:= P_1  \frac{\partial }{\partial \zeta}(\H x),\\
\D(A)&=\{ x\in X \mid \H x \in H^1((0,1);\C^2) \text{ and } \begin{bmatrix} W_1 & 
W_0\end{bmatrix}\begin{bmatrix}
 (\H x)(1)\\
 (\H x)(0)
\end{bmatrix}=0 \}
\end{align*}
and remind that due to \cite[Theorem 13.2.2]{JacZwa12} the system has a unique mild solution if $A$ generates a $C_0$-semigroup.
To show that the port-Hamiltonian operator $A$ is a \riesz, it is sufficient due to Theorem \ref{main} to prove that $A$ generates a $C_0$-group. So we have only to check that 
$W_1\H(1)Z^+(1)\oplus W_0\H(0)Z^-(0)=\C^d$ and $W_1\H(1)Z^-(1)\oplus W_0\H(0)Z^+(0)=\C^d$, where $Z^+(\zeta)$ denotes the span of the eigenvectors of $P_1\H(\zeta)$ corresponding to the positive eigenvalues of 
$P_1\H(\zeta)$ and $Z^-(\zeta)$ is the span of the eigenvectors of $P_1\H(\zeta)$ corresponding to the negative eigenvalues 
of 
$P_1\H(\zeta)$.
Defining $\gamma =\sqrt{ T(\zeta)/\rho ( \zeta) }$,  the matrix function $P_1 {\cal H}$ can be factorized as
\[    P_1 {\cal H} = \underbrace{\begin{bmatrix} \gamma & -\gamma \\ \rho^{-1} & \rho^{-1} \end{bmatrix}}_{S^{-1}} 
\underbrace{\begin{bmatrix} \gamma & 0 \\ 0 & -\gamma \end{bmatrix}}_\Delta
\underbrace{\begin{bmatrix} (2\gamma)^{-1} & \rho/2 \\  (2\gamma)^{-1} & \rho/2 \end{bmatrix}}_S. \]

It is easy to see that $Z^+(\zeta)=\Span \begin{bsmallmatrix} T(\zeta)\\\gamma(\zeta) \end{bsmallmatrix}$ and $Z^-(\zeta)=\Span \begin{bsmallmatrix} -T(\zeta)\\\gamma(\zeta) \end{bsmallmatrix}$. Then it holds for $\kappa \neq -\frac{T(1)}{\gamma(1)}$ and $\kappa \neq \frac{T(1)}{\gamma(1)}$
\begin{align*}
  W_1\H(1)Z^+(1)\oplus W_0\H(0)Z^-(0)&=
 \begin{bmatrix} 0\\ \kappa \gamma(1)+T(1)\end{bmatrix} \oplus \begin{bmatrix} -\gamma(0) \\0 \end{bmatrix}=\C^{2}, 
\end{align*}
\begin{align*}
W_1\H(1)Z^-(1)\oplus W_0\H(0)Z^+(0)
&=\begin{bmatrix} 0\\-\kappa \gamma(1)+T(1)\end{bmatrix} \oplus \begin{bmatrix} \gamma(0) \\0 \end{bmatrix}=\C^{2}, 
\end{align*}

Thus, $A$ is a \riesz for $\kappa \neq -\frac{T(1)}{\gamma(1)}$ and $\kappa \neq \frac{T(1)}{\gamma(1)}$.

In contrast to \cite{XuWei11} we do not need  determine the eigenvalues exactly, which is only possible if $\rho$ and $T$
are constant.
For spacial varying 
coefficients like $\rho(\zeta)=e^{\zeta}$ and $T(\zeta)=\zeta+1$ the problem is not analytically solvable, see 
\cite[Exercise 12.1]{JacZwa12}.

\subsection{Timoshenko-beam}
In \cite{JacZwa12} it is shown that the Timoshenko beam with boundary damping can be formulated as port-Hamiltonian system: It can be written in the form of \eqref{eqn:pde}, with
$P_1=\begin{bsmallmatrix}0 &1&0&0\\ 
1&0&0&0\\ 
0 &0&0&1\\ 
0 &0&1&0\end{bsmallmatrix}, \H(\zeta)=\begin{bsmallmatrix} K(\zeta) &0&0&0\\ 
0&\frac{1}{\rho(\zeta)}&0&0\\ 
0 &0&EI(\zeta)&1\\ 
0 &0&0& \frac{1}{I_\rho(\zeta)}\end{bsmallmatrix} $and
$P_0=\begin{bsmallmatrix}0 &0&0&-1\\ 
0&0&0&0\\ 
0 &0&0&0\\ 
1 &0&0&0 \end{bsmallmatrix}$,
where the $K(\zeta)$ denotes the shear modulus, $EI(\zeta)$ is the product of Young's modulus of elasticity and the moment of inertia of a cross section, $\rho(\zeta)$ is the mass per unit length and $I_{\rho}(\zeta)$ denotes the rotary moment of inertia of a cross section.  All these physical parameters are positive and continuously differentiable functions of $\zeta$.
To model the fact that the beam is clamped in at $\zeta=0$ and controlled at
$\zeta=1$ by the force and moment feedback, we add the boundary condition
$\begin{bsmallmatrix}0\\ 0\\ 0 \\ 0 \end{bsmallmatrix}=\begin{bsmallmatrix} W_1 & W_0\end{bsmallmatrix}\begin{bsmallmatrix} ({\cal H}x)(1,t) \\ ({\cal
         H}x)(0,t)\end{bsmallmatrix}$ with
$\begin{bsmallmatrix} W_1 & W_0\end{bsmallmatrix}=
\begin{bsmallmatrix}
0&0&0&0 &0&1&0&0\\
0&0&0&0 &0&0&0&1\\
1&\alpha_1&0&0&0&0&0&0\\
0&0&1&\alpha_2&0&0&0&0\\
\end{bsmallmatrix}$
and  $\alpha_1$, $\alpha_2$ are given positive gain feedback constants. 

Xu and Feng dedicate the paper \cite{XuFen02} to this example and they proved under the extra assumption that all physical constants are independent of $\zeta$ that the eigenvectors and generalized eigenvectors of the operator form a Riesz basis. This example is also revisited in \cite{Vil07} using another approach.
Using our main theorem, we can easy verify that the associated system operator is a \riesz.

\section{Closing remarks and open problems}

We have shown that a port-Hamil\-tonian system of the form (\ref{eqn:pde0}) is a Riesz spectral operator if and only if it generates a $C_0$-group. Many (hyperbolic) systems can be written into this form, with as main exception the Euler-Bernoulli beam equation. Of course the  basis property of this equation is well-studied, and many results are known, see e.g.\  \cite{GuoWan19}. However, we assert that the main result of this paper  does not hold for the Euler-Bernoulli beam equation.

In Theorem \ref{main}  we have shown that if the port-Hamil\-tonian systems (\ref{eqn:pde0}) is a Riesz spectral operator, then the eigenvalues (counted according to the algebraic multiplicity) can be decomposed into finitely many sets each having a uniform gap. If we count the eigenvalues without multiplicity, then \cite[Theorem 2]{JacZwa01b} shows that the eigenvalues  can be decomposed into at most $d$ sets each having a uniform gap.
We claim that this results holds true if we count the eigenvalues  according to the algebraic multiplicity.

One may ask whether the main theorem of this paper (Theorem \ref{main}) holds if we drop the assumption that $P_0$ is skew-symmetric. 
Our proof uses the fact that every well-posed port-Hamiltonian control system \eqref{eqn:pde} is exactly controllable in finite time and this property is only known in the case that $P_0$ is skew-symmetric.  Thus, if $P_0$ is an arbitrary $d\times d$-matrix, then our proof carries over to this more general case, provided we add the assumption that the corresponding port-Hamiltonian system is exactly controllable in finite-time. However, we assert that even when $P_0$ is not skew-symmetric, the system \eqref{eqn:pde} is exactly controllable in finite time, and thus this extra assumption would not be needed.

B.-Z.\ Guo and J.-M.\ Wang \cite{GuoWan19} studied the Riesz basis property for a closely related class of systems, that is, hyperbolic systems of the form  $\frac{\partial x}{\partial t} = K(\zeta) \frac{\partial x}{\partial \zeta} + C(\zeta) x$ with $K$ and $C$ diagonal. They showed that for their class of systems the state space can be split into two parts, one part generated by a $C_0$-group and the other generated by an operator without spectrum, see \cite[Theorem 4.10]{GuoWan19}.  In particular, the spectrum of their operators always lie in a strip parallel to the imaginary axis. This result does not generalize to our system class as the following example shows.
\begin{example}
We consider the port-Hamiltonian system
\begin{align*}
  \frac{\partial }{\partial t} \begin{bmatrix} x_1(\zeta,t) \\ x_2(\zeta,t) \end{bmatrix} &= 
  \begin{bmatrix} 1 & 0 \\ 0 & \frac{1}{2}  \end{bmatrix} \frac{\partial }{\partial \zeta} \begin{bmatrix} x_1(\zeta,t) \\ x_2(\zeta,t) 
\end{bmatrix}
  +  \begin{bmatrix} 0 & 1 \\ -1 & 0  \end{bmatrix}  
  \begin{bmatrix} x_1(\zeta,t) \\ x_2(\zeta,t) 
\end{bmatrix},\\
0&=\begin{bmatrix}
1 & 0 & 0&1 \\
0 & 1 & 0 &0
\end{bmatrix}
\begin{bmatrix}
 x(1,t)\\ x(0,t)
\end{bmatrix}.
\end{align*}
The system operator generates a $C_0$-semigroup, but there exists a sequence of eigenvalues which real parts converge to $-\infty$.
\end{example}

The following example shows that the equivalence $1)\Leftrightarrow 2)$ in Theorem \ref{main} does not hold for generators $A$ of $C_0$-semigroups $(T(t))_{t\ge 0}$ on Hilbert spaces even if we additionally assume that there exists admissible control operator $B\in\BL(\C^d,X_{-1})$ for $(T(t))_{t\ge 0}$ such that the control system $\dot{x}(t)=Ax(t)+Bu(t)$ is exactly controllable in finite time.
\begin{example}
Let  $A:\D(A)\subset \ell^2\to\ell^2$ be defined by $(Ax)_n=(s_n x_n)_n$, $(s_n)_{n\in \N}=(-2^n)_{n\in \N}$, and
$\D(A)=\{x\in \ell^2(\N)\mid \sum_{n\in \N}(1+\abs{s_n}^2)\abs{x_n}^2<\infty\}$. Clearly, $A$ is a discrete Riesz spectral operator, generates a $C_0$-semigroup, but not a $C_0$-group.
Here $X_{-1}=\ell^2_{-1}=\{(x_n)_{n\in\N}\mid\sum_{n\in \N}\frac{\abs{x_n}^2}{(1+\abs{s_n}^2)}<\infty\}$.
Hence, we can identify $B\in \BL(\mathbb C, \ell^2_{-1})$ with a sequence $(b_n)\in \ell^2_{-1}$. Let $(b_n)_{n\in\N}=(\sqrt{2^n})_{n\in\N}$.
Then it holds $(b_n)_{n\in\N}\in \ell^2_{-1}$.
The  Carleson measure criterion of Weiss \cite{Wei88} implies the admissibility  of $B$. 
Exact controllability in finite time follows from \cite[Theorem 2]{JacZwa01}.
\end{example}


\bibliographystyle{alphaabrrv}
\bibliography{literaturabbNEU}

\end{document}